\documentclass[12pt]{article}

\usepackage{amsmath}
\usepackage{amsthm}
\usepackage{amsfonts}
\usepackage{mathrsfs}
\usepackage{stmaryrd}
\usepackage{setspace}
\usepackage{fullpage}
\usepackage{amssymb}
\usepackage{breqn}
\usepackage{enumitem}
\usepackage{bbold}
\usepackage{authblk}
\usepackage{comment}
\usepackage{hyperref}
\usepackage{pgf,tikz}
\usepackage{graphicx}
\usepackage{subcaption}

\newtheorem{thm}{Theorem}[section]
\newtheorem{lem}[thm]{Lemma}

\newcommand\ex{\ensuremath{\mathrm{ex}}}
\newcommand\cA{{\mathcal A}}
\newcommand\cB{{\mathcal B}}

\newcommand\cE{{\mathcal E}}
\newcommand\cF{{\mathcal F}}
\newcommand\cG{{\mathcal G}}
\newcommand\cH{{\mathcal H}}

\newcommand\cM{{\mathcal M}}

\newcommand\cR{{\mathcal R}}
\newcommand\cS{{\mathcal S}}

\newcommand{\norm}[1]{\left\|#1\right\|}

\begin{document}

\title{The $(t,p)$-Norm in Classical Extremal Problems}

\author{
Xiamiao Zhao \thanks{Department of Mathematical Sciences, Tsinghua University, Beijing 100084, China.\\ Email: \texttt{zxm23@mails.tsinghua.edu.cn}.}
\hspace{0.2em}
Yuanpei Wang \thanks{ \textit{Corresponding author}. Department of Mathematics, Shanghai University, Shanghai 200444, P.R. China.\\Email:
\small \texttt{boyuan@shu.edu.cn}}

}

\date{}
\maketitle

\begin{abstract}
Given integers $r>t\ge1$ and a real number $p>0$, the $(t,p)$-norm $\norm{\cH}_{t,p}$ of an $r$-graph $\cH$ is the sum of the $p$-th powers of the degrees $d_{\cH}(T)$ over all $t$-subsets $T\subseteq V(\cH)$.  When $t=r-1$, this is the codegree $p$-norm.  For all sufficiently large $n$, we obtain the following results.  The first two apply in both the convex range $p>1$ and the concave range $0<p<1$.  First, for $r$-graphs with matching number at most $s$, we determine the maximum $(t,p)$-norm.  Second, for $k$-intersecting families, we establish an Erd\H{o}s--Ko--Rado-type theorem for the $(t,p)$-norm.  Third, for $P_\ell^r$-free hypergraphs, we determine the maximum $(t,p)$-norm for every $1\le t\le r-1$ and $p>1$.  In each of the three settings, we also characterize all extremal families.
\end{abstract}



\section{Introduction}

Let $[n]=\{1,2,\ldots,n\}$ and let $\binom{[n]}r$ denote the collection of all $r$-subsets of $[n]$.  An \emph{$r$-uniform hypergraph}, or simply an \emph{$r$-graph}, is a family $\cH\subseteq\binom{[n]}r$.  For $T\subseteq[n]$, write
\[
   d_{\cH}(T)=|\{H\in\cH:T\subseteq H\}|
\]
for the \emph{degree} of $T$ in $\cH$.  Given $1\le t\le r-1$ and $p>0$, we define the \emph{$(t,p)$-norm} of $\cH$ by
\begin{equation*}
   \norm{\cH}_{t,p}:=\sum_{T\in\binom{[n]}t}d_{\cH}(T)^p.
\end{equation*}
When $t=r-1$, this is the $p$-norm of the codegree vector.  Such power sums refine ordinary edge-counting extremal problems, since
\[
   \norm{\cH}_{t,1}=\binom rt |\cH|.
\]
Thus the case $p=1$ is exactly the usual Tur\'an problem, whereas the cases $p>1$ and $0<p<1$ exhibit different behavior: convexity favors concentration of degree mass, while concavity favors a more spread-out distribution.

Extremal set theory asks how large a family of finite sets can be under restrictions on matchings, intersections, or forbidden substructures.  The classical theorem of Erd\H{o}s, Ko and Rado \cite{ErdosKoRado1961} determines the maximum size of an intersecting subfamily of $\binom{[n]}r$ for $n$ sufficiently large, and Hilton and Milner \cite{HiltonMilner1967} gave the first non-trivial stability version.  Wilson \cite{Wilson1984} later obtained the exact bound in the Erd\H{o}s--Ko--Rado theorem, while the complete intersection theorem of Ahlswede and Khachatrian \cite{AhlswedeKhachatrian1997} determined the maximum size of a $k$-intersecting family for all parameter ranges.

Another central problem is the Erd\H{o}s matching problem.  For an $r$-graph $\cH$, the \emph{matching number} $\nu(\cH)$ is the maximum number of pairwise disjoint edges in $\cH$.  Erd\H{o}s \cite{Erdos1965} conjectured the exact maximum number of edges in an $r$-graph with a given matching number.  The graph case follows from the theorem of Erd\H{o}s and Gallai \cite{ErdosGallai1959}.  For uniform hypergraphs the conjecture remains open in general, although it is known in several important ranges; see, for example, Frankl \cite{Frankl2013,Frankl2017}, Frankl and Kupavskii \cite{FranklKupavskii2019}, and the stability theorem of Bollob\'as, Daykin and Erd\H{o}s \cite{BollobasDaykinErdos1976}.  The shifting method, surveyed by Frankl \cite{Frankl1987}, is one of the basic tools in this area.

Recently, several authors have studied extremal problems in which the objective is a power sum of degrees or codegrees rather than the number of edges.  Balogh, Clemen and Lidick\'y \cite{BaloghSurvey2021,BaloghClemenLidicky2022} studied hypergraph Tur\'an problems under the $\ell_2$-norm of the codegree vector, and Brooks and Linz \cite{BrooksLinz2026} obtained exact and asymptotic results for related problems.  Chen, I\v{l}kovi\v{c}, Le\'on, Liu and Pikhurko \cite{ChenIlkovicLeonLiuPikhurko2024} introduced a general framework for non-degenerate Tur\'an problems under $(t,p)$-norms.  Very recently, Cao, Lu and Zhang \cite{CaoLuZhangDegreePowers2026} studied degree powers in intersecting families.  More precisely, for every $p\ge2$, they proved that a full $k$-star maximizes $\norm{\cH}_{r-1,p}$ among $k$-intersecting $r$-graphs whenever $n\ge(k+1)(r-k+1)$.  In the intersecting case $k=1$, they further proved that, for every $p\ge2$ and every $1\le t\le r-1$, a full point-star maximizes $\norm{\cH}_{t,p}$ whenever $n\ge2r$.  Thus, their results hold throughout the sharp classical Erd\H{o}s--Ko--Rado ranges, rather than only for sufficiently large $n$.  Related codegree-density questions appear in the work of Mubayi and Zhao \cite{MubayiZhao2007}; see Keevash's survey \cite{Keevash2011} for the broader background on hypergraph Tur\'an theory.

The bounded matching number analogue is closely related to sunflower-counting problems.  The sunflower lemma of Erd\H{o}s and Rado \cite{ErdosRado1960} initiated a long line of work on set systems with prescribed intersections.  Wang and Peng \cite{WangPeng2026}, Zhang, Cao and Lu \cite{ZhangCaoLu2026}, and Zhou and Yuan \cite{ZhouYuan2026} studied sunflower-counting and norm versions of matching-free extremal problems.  In the graph case, the analogous connection between degree powers and star counts was first exploited by Gerbner \cite{GerbnerDegreePowers2025}.  In the codegree case $t=r-1$, certain codegree power sums can be interpreted as counts of sunflowers with prescribed core.  For general $t$, such an exact sunflower interpretation is no longer available, but the same extremal star phenomenon persists.

The last part of the paper concerns expansion paths.  Let $P_\ell^r$ denote the \emph{$r$-uniform linear path} with $\ell$ edges.  Kostochka, Mubayi and Verstra\"ete \cite{KostochkaMubayiVerstraete2015} determined $\ex_r(n,P_\ell^r)$ exactly for $r\ge3$, $\ell\ge4$, and $n$ sufficiently large, and characterized the extremal examples.  Their proof proceeds through asymptotic, stability and exact steps, using shadows and full subgraphs.  We apply their edge-stability theorem both to the original hypergraph and, for $t\ge3$, to a high $t$-shadow.  The low-dimensional cases require separate weighted arguments.  For $t=2$ we define a graph of high pairs and use a local weighted estimate showing that, except for $O(n^{r-2})$ exceptional hyperedges, each hyperedge contains at most $r-1$ high pairs.  For $t=1$ we use an analogous weighted high-vertex estimate.  These estimates give the sharp asymptotic norms and then, via the edge-stability theorem of Kostochka, Mubayi and Verstra\"ete, the required stability.

\subsection{Main results}

For an integer $s\ge1$, define
\[
   \cH_{n,r,s}:=\{H\in\binom{[n]}r:H\cap[s]\ne\emptyset\}.
\]
This family has matching number at most $s$.
For fixed $t,p$ and a background parameter $c\ge0$, write
\begin{equation*}
   \Phi_{c,t}(\cH):=\sum_{T\in\binom{[n]}t}\left(c+d_{\cH}(T)\right)^p .
\end{equation*}
Thus, $\Phi_{0,t}(\cH)=\norm{\cH}_{t,p}$.
First, we consider the maximum $(t,p)$-norm in the bounded matching number setting.
We will separate the cases $p>1$ and $0<p<1$.
\begin{thm}\label{thm:matching-convex}
Let $r,s,t$ be positive integers with $1\le t\le r-1$, and let $p>1$. There exists $N_+(r,s,t)$ such that for all $n\ge N_+(r,s,t)$, simultaneously for every $c\ge0$, if $\cH\subseteq\binom{[n]}r$ satisfies $\nu(\cH)\le s$, then
\[
   \Phi_{c,t}(\cH)\le \Phi_{c,t}(\cH_{n,r,s}).
\]
Moreover, equality holds if and only if $\cH$ is isomorphic to $\cH_{n,r,s}$.
\end{thm}

\begin{thm}\label{thm:matching-concave}
Let $r,s,t$ be positive integers with $1\le t\le r-1$, and let $0<p<1$. There exists $N_-(r,s,t,p)$ such that for all $n\ge N_-(r,s,t,p)$, if $\cH\subseteq\binom{[n]}r$ satisfies $\nu(\cH)\le s$, then
\[
   \norm{\cH}_{t,p}\le \norm{\cH_{n,r,s}}_{t,p}.
\]
Moreover, equality holds if and only if $\cH$ is isomorphic to $\cH_{n,r,s}$.
\end{thm}

Taking $c=0$ in Theorem~\ref{thm:matching-convex} gives the ordinary $(t,p)$-norm statement.  The exact value of the extremal norm in Theorems~\ref{thm:matching-convex} and~\ref{thm:matching-concave} is
\begin{equation*}
\norm{\cH_{n,r,s}}_{t,p}=
\sum_{i=0}^{\min\{s,t\}}
\binom{s}{i}\binom{n-s}{t-i}
\left(\binom{n-t}{r-t}-\mathbb{1}_{\{i=0\}}\binom{n-s-t}{r-t}\right)^p .
\end{equation*}
The exact calculation of this norm is given in the next section.
For $p=1$, the assertion follows directly from Frankl's theorem, since $\norm{\cH}_{t,1}=\binom rt|\cH|$.
We separate the cases $p>1$ and $0<p<1$ because they has different proofs.

For $1\le k\le r-1$, define the \emph{full $k$-star}
\[
   \cS_{n,r,k}:=\{H\in\binom{[n]}r:[k]\subseteq H\}.
\]
As in the bounded matching number setting, we give the background form when $p>1$ and focus on the $(t,p)$-norm when $0<p<1$.
\begin{thm}\label{thm:ekr-convex}
Let $1\le k\le r-1$, $1\le t\le r-1$, and let $p>1$.  For all sufficiently large $n$, simultaneously for every $c\ge0$, if $\cH\subseteq\binom{[n]}r$ is $k$-intersecting, then
\[
   \Phi_{c,t}(\cH)\le \Phi_{c,t}(\cS_{n,r,k}).
\]
Moreover, equality holds if and only if $\cH$ is isomorphic to $\cS_{n,r,k}$.
\end{thm}

\begin{thm}\label{thm:ekr-concave}
Let $1\le k\le r-1$, $1\le t\le r-1$, and let $0<p<1$. For all sufficiently large $n$, if $\cH\subseteq\binom{[n]}r$ is $k$-intersecting, then
\[
   \norm{\cH}_{t,p}\le \norm{\cS_{n,r,k}}_{t,p}.
\]
Moreover, equality holds if and only if $\cH$ is isomorphic to $\cS_{n,r,k}$.
\end{thm}

For the full $k$-star, the extremal value in Theorem~\ref{thm:ekr-convex} is
\begin{equation*}
\Phi_{c,t}(\cS_{n,r,k})=
\sum_{i=0}^{\min\{k,t\}}
\binom{k}{i}\binom{n-k}{t-i}
\left(c+\binom{n-k-t+i}{r-k-t+i}\right)^p,
\end{equation*}
where a binomial coefficient with a negative lower index is interpreted as zero.  Taking $c=0$ gives
\begin{equation*}
\norm{\cS_{n,r,k}}_{t,p}=
\sum_{i=0}^{\min\{k,t\}}
\binom{k}{i}\binom{n-k}{t-i}
\binom{n-k-t+i}{r-k-t+i}^p.
\end{equation*}
The exact calculation of this norm is given in the next section.
The case $p=1$ is again the ordinary edge-extremal theorem.

Finally we state the expansion path results. 
In this part, we use more \textit{hypergraph} language.
Thus, we redefine the following hypergraph.
For a fixed set $A\subseteq[n]$, write
\[
   \cH_{n,r,A}:=\{H\in\binom{[n]}r:H\cap A\ne\emptyset\},
\]
and write $\cH_{n,r,a}$ when $|A|=a$ and the choice of $A$ is irrelevant.

In the even case $\ell=2s$, the exact extremal family contains a lower-order part outside the main star.  Fix an $(s-1)$-set $A\subseteq[n]$ and a pair $Q\subseteq[n]\setminus A$, and define
\begin{equation}\label{eq:even-extremal-family}
   \cE_{n,r,s}(A,Q):=
   \{E\in\binom{[n]}r:E\cap A\ne\emptyset\}
   \cup
   \{E\in\binom{[n]\setminus A}r:Q\subseteq E\}.
\end{equation}
When $A$ and $Q$ are immaterial, we write $\cE_{n,r,s}$.  Define $N=n-s+1$ and $D_t=\binom{n-t}{r-t}$.  Then
\begin{align}\label{eq:even-extremal-value}
\norm{\cE_{n,r,s}}_{t,p}
&=
\sum_{i=1}^{\min\{s-1,t\}}
\binom{s-1}{i}\binom{N}{t-i}D_t^p\notag\\
&\quad+
\sum_{j=0}^{\min\{2,t\}}
\binom2j\binom{N-2}{t-j}
\left(
D_t-\binom{N-t}{r-t}
+\binom{N-t-2+j}{r-t-2+j}
\right)^p,
\end{align}
where a binomial coefficient with a negative lower index is interpreted as zero.

\begin{thm}\label{thm:exact-odd-path}
Let $r\ge3$, $a\ge2$, $1\le t\le r-1$, and $p>1$ be fixed.  Then, for all sufficiently large $n$, every $P_{2a+1}^r$-free family $\cH\subseteq\binom{[n]}r$ satisfies
\[
   \norm{\cH}_{t,p}\le \norm{\cH_{n,r,a}}_{t,p}.
\]
Equality holds if and only if $\cH$ is isomorphic to  $\cH_{n,r,a}$.
\end{thm}

\begin{thm}\label{thm:exact-even-path}
Let $r\ge3$, $s\ge2$, $1\le t\le r-1$, and $p>1$ be fixed.  Then, for all sufficiently large $n$, every $P_{2s}^r$-free family $\cH\subseteq\binom{[n]}r$ satisfies
\[
   \norm{\cH}_{t,p}\le \norm{\cE_{n,r,s}}_{t,p}.
\]
Equality holds if and only if $\cH$ is isomorphic to $\cE_{n,r,s}(A,Q)$ for some $(s-1)$-set $A$ and some pair $Q\subseteq[n]\setminus A$.
\end{thm}

\subsection{Organization}

Section~\ref{sec:prelim} collects standard tools and norm calculations.  The remainder of the paper is organized according to the three extremal problems studied here.  Section~\ref{sec:matching} treats bounded matching number, proving the convex and concave ranges in parallel subsections.  Section~\ref{sec:ekr} proves the background EKR-type theorem for $p>1$ and the zero-background concave EKR theorem.  Section~\ref{sec:paths} treats expansion paths: first the auxiliary shadow, weighted-pair and weighted-vertex estimates, then the asymptotic and stability theorem, and finally the exact odd and even path results obtained from stability.

\section{Preliminaries}\label{sec:prelim}

A \emph{matching} is a collection of pairwise disjoint edges. The matching number of $\cH$ is denoted by $\nu(\cH)$. For $T\subseteq[n]$, the \emph{link} of $T$ in $\cH$ is $\cH(T):=\{H\setminus T:H\in\cH,\ T\subseteq H\}$. Thus, $d_{\cH}(T)=|\cH(T)|$.
For a graph $G$ and a set $S\subseteq V(G)$, let $e_G(S)$ be the number of edges of $G$ contained in $S$. For two families $\cF$ and $\cG$, write $\cF\triangle\cG:=(\cF\setminus\cG)\cup(\cG\setminus\cF)$ for their \emph{symmetric difference}.

An $r$-graph is \emph{$k$-intersecting} if $|E\cap F|\ge k$ for all $E,F\in\cH$. It is \emph{non-trivial} if it is not contained in a full $k$-star. The \emph{$r$-uniform linear path} $P_\ell^r$ consists of edges $E_1,\ldots,E_\ell$ such that $|E_i\cap E_{i+1}|=1$ for every $i$, and all other intersections are empty.

We shall use the Erd\H{o}s matching theorem in the following form.

\begin{thm}[Frankl~\cite{Frankl2013}]\label{thm:frankl}
Let $r,s$ be positive integers. If $n\ge(2s+1)r-s$ and $\cH\subseteq\binom{[n]}r$ satisfies $\nu(\cH)\le s$, then
\[
   |\cH|\le \binom nr-\binom{n-s}r.
\]
Moreover, equality holds only when $\cH$ is isomorphic to $\cH_{n,r,s}$.
\end{thm}

We also use the EKR theorem for sufficiently large $n$ together with the standard non-trivial estimate.  For the latter, see Ahlswede and Khachatrian~\cite{AhlswedeKhachatrian1996}; when $k=1$, this is the Hilton--Milner theorem~\cite{HiltonMilner1967}.

\begin{thm}[Erd\H{o}s--Ko--Rado~\cite{ErdosKoRado1961}]\label{thm:nontrivial-ekr}
Fix $1\le k\le r-1$. For all sufficiently large $n$, every $k$-intersecting family $\cH\subseteq\binom{[n]}r$ satisfies
\[
   |\cH|\le \binom{n-k}{r-k},
\]
with equality only for a full $k$-star. Moreover, if $\cH$ is not contained in any full $k$-star, then
\[
   |\cH|=O(n^{r-k-1}).
\]

\end{thm}

We shall use the following asymptotic and stability theorem of Kostochka, Mubayi and Verstra\"ete on linear paths.  We apply it directly to the original hypergraph and, when the shadow uniformity is at least three, to a high shadow.

\begin{thm}[Kostochka--Mubayi--Verstra\"ete\cite{KostochkaMubayiVerstraete2015}]\label{thm:kmv-edge-stability}
Fix $r\ge3$ and $\ell\ge4$, and let $a=\left\lfloor(\ell-1)/2\right\rfloor$.
As $n\to\infty$,
\[
   \ex_r(n,P_\ell^r)=(a+o(1))\binom{n}{r-1}.
\]
Moreover, if a sequence of $P_\ell^r$-free families $\cH\subseteq\binom{[n]}r$ satisfies
\[
   |\cH|\ge (a-o(1))\binom{n}{r-1},
\]
then there is an $a$-set $A\subseteq[n]$ such that
\[
   |\cH\triangle \cH_{n,r,A}|=o(n^{r-1}).
\]
\end{thm}
The following is the stability ingredient used for the concave matching-number range.

\begin{lem}\label{lem:matching-stability}
Let $r,s$ be positive integers. There is a constant $C=C(r,s)$ such that the following holds. Let $\cH\subseteq\binom{[n]}r$ satisfy $\nu(\cH)\le s$, and suppose that $\cH$ is not contained in any family of the form
\[
   \{H\in\binom{[n]}r:H\cap S\ne\emptyset\},\qquad S\in\binom{[n]}s.
\]
Then there is a set $D\subseteq[n]$ with $|D|\le s-1$ such that
\[
   |\{H\in\cH:H\cap D=\emptyset\}|\le Cn^{r-2}.
\]
\end{lem}

\begin{proof}
Choose a constant $M=M(r,s)$ sufficiently large, and define
\[
   D=\{v\in[n]:d_{\cH}(v)>Mn^{r-2}\}.
\]
We first show that $|D|\le s-1$. Suppose that $D$ contains an $s$-set $S=\{v_1,\ldots,v_s\}$. Since $\cH$ is not contained in the $s$-star centered at $S$, there exists $B\in\cH$ with $B\cap S=\emptyset$. We greedily choose pairwise disjoint edges $E_1,\ldots,E_s$ such that $v_i\in E_i$ and each $E_i$ is disjoint from $B$. At step $i$, the number of edges containing $v_i$ and meeting the bounded set $B\cup E_1\cup\cdots\cup E_{i-1}\cup(S\setminus\{v_i\})$ is $O(n^{r-2})$. Taking $M$ sufficiently large, an admissible choice remains. This yields $s+1$ pairwise disjoint edges, a contradiction.

Let $\cR=\{H\in\cH:H\cap D=\emptyset\}$. Every vertex outside $D$ has degree at most $Mn^{r-2}$ in $\cH$, and hence also in $\cR$. If $|\cR|>(s+1)rMn^{r-2}$, then we can greedily choose $s+1$ pairwise disjoint edges in $\cR$, again a contradiction. Hence $|\cR|=O(n^{r-2})$.
\end{proof}

We shall use the following norm calculations for the two star constructions.

\begin{lem}\label{lem:vertex-star-norm}
Let $1\le t\le r-1$, $p>0$, and let $s\ge0$ be fixed. Then
\[
\norm{\cH_{n,r,s}}_{t,p}=
\sum_{i=0}^{\min\{s,t\}}
\binom{s}{i}\binom{n-s}{t-i}
\left(\binom{n-t}{r-t}-\mathbb{1}_{\{i=0\}}\binom{n-s-t}{r-t}\right)^p .
\]
If $0<p<1$, then
\[
\norm{\cH_{n,r,s}}_{t,p}=
\left(\frac{s^p}{t!(r-t-1)!^p}+o(1)\right)n^{t+p(r-t-1)}.
\]
\end{lem}

\begin{proof}
For $T\in\binom{[n]}t$, let $i=|T\cap[s]|$. There are $\binom{s}{i}\binom{n-s}{t-i}$ such $T$. If $i\ge1$, then every $r$-set containing $T$ lies in $\cH_{n,r,s}$. If $i=0$, one must exclude the $r$-sets containing $T$ and avoiding $[s]$. This gives the exact formula. The asymptotic formula for $0<p<1$ comes from the $t$-sets disjoint from $[s]$; the $t$-sets meeting $[s]$ contribute only $O(n^{t-1+p(r-t)})$, which is lower order.
\end{proof}

\begin{lem}\label{lem:k-star-norm}
Let $1\le k\le r-1$, $1\le t\le r-1$, and $p>0$. Then
\[
\norm{\cS_{n,r,k}}_{t,p}=
\sum_{i=0}^{\min\{k,t\}}
\binom{k}{i}\binom{n-k}{t-i}
\binom{n-k-t+i}{r-k-t+i}^p.
\]
If $p>1$, then
\[
\norm{\cS_{n,r,k}}_{t,p}=
\begin{cases}
\displaystyle \left(\frac{\binom{k}{t}}{(r-k)!^p}+o(1)\right)n^{p(r-k)}, & 1\le t<k,\\[1em]
\displaystyle \left(\frac{1}{(t-k)!(r-t)!^p}+o(1)\right)n^{t-k+p(r-t)}, & k\le t\le r-1.
\end{cases}
\]
If $0<p<1$, let $i_0=\max\{0,k+t-r\}$. Then
\[
\norm{\cS_{n,r,k}}_{t,p}=
\left(\frac{\binom{k}{i_0}}{(t-i_0)!(r-k-t+i_0)!^p}+o(1)\right)
 n^{t-i_0+p(r-k-t+i_0)}.
\]
\end{lem}

\begin{proof}
Let $K=[k]$. If $|T\cap K|=i$, then an edge of the full $k$-star containing $T$ must contain $T\cup K$, and the degree is $\binom{n-k-t+i}{r-k-t+i}$. This gives the exact formula. The asymptotics follow by comparing the exponents of the summands: for $p>1$ the largest feasible $i$ dominates, while for $0<p<1$ the smallest feasible $i$ dominates.
\end{proof}

We shall use the following elementary inequalities. If $p>1$ and $x_i\ge0$, then
\begin{equation*}
   \sum_i x_i^p\le \left(\max_i x_i\right)^{p-1}\sum_i x_i.
\end{equation*}
If $0<p<1$, then
\begin{equation}\label{eq:subadditive}
   (x+y)^p\le x^p+y^p
\end{equation}
for all $x,y\ge0$, and Jensen's inequality gives
\begin{equation}\label{eq:jensen}
   \sum_{i=1}^m x_i^p\le m^{1-p}\left(\sum_{i=1}^m x_i\right)^p.
\end{equation}

\section{Bounded matching number}\label{sec:matching}

\subsection{The convex range}\label{sec:matching-convex}

We prove Theorem~\ref{thm:matching-convex} by shifting. For $1\le x<y\le n$, let $S_{xy}$ be the usual \emph{shift}. It replaces an edge $H$ containing $y$ but not $x$ by $(H\setminus\{y\})\cup\{x\}$ if the latter edge is not already present, and otherwise keeps $H$. Shifting does not increase matching number.

\begin{lem}\label{lem:shift-norm}
Let $1\le t\le r-1$ and $p\ge1$. For every $r$-graph $\cH\subseteq\binom{[n]}r$ and every $1\le x<y\le n$,
\[
   \norm{S_{xy}(\cH)}_{t,p}\ge \norm{\cH}_{t,p}.
\]
More generally, for every $c\ge0$,
\[
   \Phi_{c,t}(S_{xy}(\cH))
   \ge
   \Phi_{c,t}(\cH).
\]
\end{lem}

\begin{proof}
The degrees of $t$-sets containing both $x,y$, or containing neither, are unchanged. It remains to compare pairs
\[
   T_x=P\cup\{x\},\qquad T_y=P\cup\{y\},
\]
where $P\in\binom{[n]\setminus\{x,y\}}{t-1}$.

For fixed $P$, compare the pairs of possible edges $P\cup\{x\}\cup R$ and $P\cup\{y\}\cup R$, where $R$ ranges over all $(r-t)$-sets disjoint from $P\cup\{x,y\}$. Let $a_0$ be the number of pairs in which only the $x$-edge appears, $b_0$ the number of pairs in which only the $y$-edge appears, and $c_0$ the number of pairs in which both appear. Let $q_0$ be the number of edges containing $P\cup\{x,y\}$. Before shifting, the two degrees are
\[
   a_0+c_0+q_0,\\ b_0+c_0+q_0.
\]
After shifting, they are
\[
   a_0+b_0+c_0+q_0,\\ c_0+q_0.
\]
The sum is unchanged and the absolute difference has increased from $|a_0-b_0|$ to $a_0+b_0$. Thus the shifted pair majorizes the original pair. Adding the same constant to both coordinates preserves majorization. Since $x\mapsto x^p$ is convex for $p\ge1$, Karamata's inequality gives the desired inequality for this pair. Summing over all $P$ proves the lemma.
\end{proof}

The background parameter $c$ records the contribution of a complete star removed during the induction.

\begin{lem}\label{lem:bad-comparison}
Let $p>1$ and $1\le q\le s$. Define
\[
   \cG_q=\{G\in\binom{[n]}r:|G\cap[qr+1]|\ge2\}.
\]
There exists $N_1=N_1(r,s,t)$ such that for all $n\ge N_1$, all $c\ge0$, and all $1\le q\le s$,
\[
   \Phi_{c,t}(\cG_q)<\Phi_{c,t}(\cH_{n,r,q}).
\]
\end{lem}

\begin{proof}
Let $D=\binom{n-t}{r-t}$ and let $g(x)=(c+x)^p-c^p$. The function $g$ is increasing and convex with $g(0)=0$. Hence $g(x)\le xg(D)/D$ for $0\le x\le D$. Since every $t$-degree in $\cG_q$ is at most $D$,
\[
   \Phi_{c,t}(\cG_q)-\binom nt c^p
   \le \frac{g(D)}D\sum_T d_{\cG_q}(T)
   = \frac{g(D)}D\binom rt |\cG_q|.
\]
On the other hand, every $t$-set containing exactly one vertex of $[q]$ has degree $D$ in $\cH_{n,r,q}$. Hence
\[
   \Phi_{c,t}(\cH_{n,r,q})-\binom nt c^p
   \ge q\binom{n-q}{t-1}g(D).
\]
It is therefore enough to compare orders. We have $|\cG_q|=O(n^{r-2})$, because each edge in $\cG_q$ contains at least two vertices from a fixed set. Meanwhile $q\binom{n-q}{t-1}D=\Theta(n^{r-1})$. Thus the desired inequality holds for all sufficiently large $n$. Since $q$ ranges over a finite set, a single $N_1(r,s,t)$ works for all $q$.
\end{proof}

\begin{proof}[Proof of Theorem~\ref{thm:matching-convex}]
We prove the assertion by induction on $s$, simultaneously for every $c\ge0$. Choose $N_+(r,s,t)$ recursively so that $N_+(r,s,t)\ge (2s+1)r-s$, $N_+(r,s,t)\ge N_+(r,s-1,t)+1$, $N_+(r,s,t)\ge(s+1)r$, and Lemma~\ref{lem:bad-comparison} applies for every $n\ge N_+(r,s,t)$. The existence of such a threshold follows from Lemma~\ref{lem:bad-comparison}.

The case $s=0$ is trivial. Let $s\ge1$ and assume the statement for $s-1$. By Lemma~\ref{lem:shift-norm}, an extremal family may be assumed shifted. We may also assume it is maximal subject to $\nu(\cH)\le s$.

We first show that $\{1,sr+2,sr+3,\ldots,(s+1)r\}\in\cH$. If not, shiftedness implies that no edge of $\cH$ meets $[sr+1]$ in at most one vertex. Hence $\cH\subseteq\cG_s$, contradicting Lemma~\ref{lem:bad-comparison} and extremality.

Next we show that $\{1,n-r+2,n-r+3,\ldots,n\}\in\cH$. Otherwise, by maximality, adding this edge creates a matching of size $s+1$. Thus $\cH$ contains $s$ disjoint edges disjoint from it. Since $\cH$ is shifted, these $s$ edges can be shifted into the interval $[2,sr+1]$. Together with the edge from the previous paragraph, this gives a matching of size $s+1$, a contradiction.

Since $\cH$ is shifted and contains $\{1,n-r+2,\ldots,n\}$, every $r$-set containing $1$ lies in $\cH$. Let
\[
   \cH_0=\{H\in\cH:1\notin H\}\subseteq\binom{[2,n]}r.
\]
Then $\nu(\cH_0)\le s-1$, for otherwise a matching of size $s$ in $\cH_0$ together with a suitable edge through $1$ would form a matching of size $s+1$.

For every $T\subseteq[2,n]$ with $|T|=t$,
\[
   d_{\cH}(T)=\binom{n-t-1}{r-t-1}+d_{\cH_0}(T),
\]
and the same identity holds for $\cH_{n,r,s}$ with $\cH_0$ replaced by $\cH_{n-1,r,s-1}$. The contribution of $t$-sets containing $1$ is also the same for $\cH$ and $\cH_{n,r,s}$. The induction hypothesis, applied to $\cH_0$ with background $c+\binom{n-t-1}{r-t-1}$, gives $\Phi_{c,t}(\cH)\le\Phi_{c,t}(\cH_{n,r,s})$.

If equality holds, then after shifting we obtain the shifted extremal family $\cH_{n,r,s}$. Shifting preserves the number of edges, so the original family has $|\cH_{n,r,s}|$ edges. Since $n\ge(2s+1)r-s$, Frankl's theorem implies that the original family is isomorphic to $\cH_{n,r,s}$.
\end{proof}

\subsection{The concave range}\label{sec:matching-concave}

\begin{lem}\label{lem:matching-nonstar-concave}
Let $1\le t\le r-1$ and $0<p<1$. If $\cH\subseteq\binom{[n]}r$ satisfies $\nu(\cH)\le s$ and is not contained in any $s$-star, then
\[
\norm{\cH}_{t,p}\le
\left(\frac{(s-1)^p}{t!(r-t-1)!^p}+o(1)\right)n^{t+p(r-t-1)}.
\]
\end{lem}

\begin{proof}
By Lemma~\ref{lem:matching-stability}, there exists $D\subseteq[n]$ with $|D|\le s-1$ such that, with $\cR=\{H\in\cH:H\cap D=\emptyset\}$, we have $|\cR|=O(n^{r-2})$. Let $\cA=\{H\in\binom{[n]}r:H\cap D\ne\emptyset\}$. For every $T$, $d_{\cH}(T)\le d_{\cA}(T)+d_{\cR}(T)$. By subadditivity \eqref{eq:subadditive}
\[
   \norm{\cH}_{t,p}\le \norm{\cA}_{t,p}+\norm{\cR}_{t,p}.
\]
Lemma~\ref{lem:vertex-star-norm} gives
\[
   \norm{\cA}_{t,p}\le
   \left(\frac{(s-1)^p}{t!(r-t-1)!^p}+o(1)\right)n^{t+p(r-t-1)}.
\]
Also,
\[
   \sum_T d_{\cR}(T)=\binom rt|\cR|=O(n^{r-2}).
\]
By Jensen's inequality \eqref{eq:jensen},
\[
   \norm{\cR}_{t,p}\le \binom nt^{1-p}\left(\binom rt|\cR|\right)^p
   =O(n^{t(1-p)+p(r-2)}).
\]
Since $t(1-p)+p(r-2)=t+p(r-t-1)-p$, this is lower order. The lemma follows.
\end{proof}

\begin{proof}[Proof of Theorem~\ref{thm:matching-concave}]
If $\cH$ is not contained in any $s$-star, then Lemma~\ref{lem:matching-nonstar-concave} and Lemma~\ref{lem:vertex-star-norm} give, for $n$ large,
\[
   \norm{\cH}_{t,p}<\norm{\cH_{n,r,s}}_{t,p},
\]
because $(s-1)^p<s^p$. If $\cH$ is contained in an $s$-star, then every $t$-degree of $\cH$ is at most the corresponding $t$-degree of the full $s$-star. Since $x\mapsto x^p$ is increasing, the result follows. Equality forces all edges of the full star to be present.
\end{proof}

\section{The EKR-type theorem}\label{sec:ekr}

We first prove the background form of the convex EKR theorem.  The parameter $c$ is kept throughout the proof, and the ordinary $(t,p)$-norm statement is recovered by taking $c=0$.
For a family $\cF$ on an $n$-element ground set, define
\[
   \Delta_{c,t}(\cF):=\Phi_{c,t}(\cF)-\binom nt c^p.
\]
The subtracted term is independent of $\cF$, so maximizing $\Phi_{c,t}$ is the same as maximizing $\Delta_{c,t}$.

We shall use two elementary estimates.  If $p>1$, $c\ge0$, and $0\le d\le M$, then
\begin{equation}\label{eq:background-upper-increment}
   (c+d)^p-c^p\le p(c+M)^{p-1}d,
\end{equation}
and for every $d\ge0$,
\begin{equation}\label{eq:background-lower-increment}
   (c+d)^p-c^p\ge d(c+d)^{p-1}.
\end{equation}
The first inequality follows from the mean value theorem, and the second from
$1-(1-z)^p\ge z$ with $z=d/(c+d)$.

\begin{lem}\label{lem:ekr-background-nontriv}
Fix $1\le k\le r-1$, $1\le t\le r-1$, and $p>1$.  Uniformly over all $c\ge0$, if $\cH\subseteq\binom{[n]}r$ is a non-trivial $k$-intersecting family, then
\[
   \Delta_{c,t}(\cH)=o\left(\Delta_{c,t}(\cS_{n,r,k})\right).
\]
\end{lem}

\begin{proof}
By Theorem~\ref{thm:nontrivial-ekr},
\[
   |\cH|=O(n^{r-k-1}).
\]

Suppose first that $t\ge k$.  Write $D_t=\binom{n-t}{r-t}$.  Since every $t$-degree is at most $D_t$, \eqref{eq:background-upper-increment} gives
\[
\Delta_{c,t}(\cH)
\le p(c+D_t)^{p-1}\sum_T d_{\cH}(T)
=O\left((c+D_t)^{p-1}n^{r-k-1}\right).
\]
On the other hand, in the full $k$-star every $t$-set containing the centre $[k]$ has degree $D_t$, and there are $\binom{n-k}{t-k}$ such $t$-sets.  By \eqref{eq:background-lower-increment},
\[
\Delta_{c,t}(\cS_{n,r,k})
\ge \binom{n-k}{t-k}\,D_t(c+D_t)^{p-1}
=\Omega\left((c+D_t)^{p-1}n^{r-k}\right).
\]
Thus the non-trivial contribution is smaller by a factor of $n$.

Now suppose that $t<k$.  For each $T\in\binom{[n]}t$, the link $\cH(T)$ is a $(k-t)$-intersecting $(r-t)$-graph.  The EKR theorem therefore gives
$d_{\cH}(T)=O(n^{r-k})$
uniformly in $T$.  Hence, by \eqref{eq:background-upper-increment},
\[
\Delta_{c,t}(\cH)
\le O\left((c+n^{r-k})^{p-1}\right)\sum_Td_{\cH}(T)
=O\left((c+n^{r-k})^{p-1}n^{r-k-1}\right).
\]
In the full $k$-star, each of the $\binom{k}{t}$ $t$-sets contained in the centre has degree
\[
   D_k=\binom{n-k}{r-k}=\Theta(n^{r-k}).
\]
Again using \eqref{eq:background-lower-increment},
\[
\Delta_{c,t}(\cS_{n,r,k})
\ge \binom{k}{t}D_k(c+D_k)^{p-1}
=\Omega\left((c+n^{r-k})^{p-1}n^{r-k}\right).
\]
This is again larger than the non-trivial bound by a factor of $n$.  The lemma follows.
\end{proof}

\begin{proof}[Proof of Theorem~\ref{thm:ekr-convex}]
If $\cH$ is not contained in any full $k$-star, Lemma~\ref{lem:ekr-background-nontriv} shows that it has strictly smaller $\Phi_{c,t}$ than the full $k$-star for all sufficiently large $n$.
If $\cH$ is contained in a full $k$-star, then every $t$-degree of $\cH$ is at most the corresponding $t$-degree of that full $k$-star.  Since $x\mapsto(c+x)^p$ is strictly increasing, the coordinatewise comparison gives the theorem, and equality forces every edge of the full $k$-star to be present.
\end{proof}

It remains to prove the concave zero-background EKR theorem.

\begin{lem}\label{lem:ekr-concave-nontriv}
Let $0<p<1$. If $\cH\subseteq\binom{[n]}r$ is a non-trivial $k$-intersecting family, then
\[
   \norm{\cH}_{t,p}=o(\norm{\cS_{n,r,k}}_{t,p}).
\]
\end{lem}

\begin{proof}
Again $|\cH|=O(n^{r-k-1})$. Let $i_0=\max\{0,k+t-r\}$, so Lemma~\ref{lem:k-star-norm} gives
\[
   \norm{\cS_{n,r,k}}_{t,p}=\Theta\left(n^{t-i_0+p(r-k-t+i_0)}\right).
\]
If $k+t\le r$, then $i_0=0$. By Jensen's inequality,
\[
\norm{\cH}_{t,p}\le \binom nt^{1-p}\left(\binom rt|\cH|\right)^p
=O(n^{t+p(r-k-t)-p}),
\]
which is lower order than the star.

If $k+t>r$, then the star has order $n^{r-k}$. Since $0<p<1$ and $d^p\le d$ for every integer $d\ge0$,
\[
   \norm{\cH}_{t,p}\le \sum_T d_{\cH}(T)=\binom rt|\cH|=O(n^{r-k-1}),
\]
again lower order.
\end{proof}

\begin{proof}[Proof of Theorem~\ref{thm:ekr-concave}]
If $\cH$ is non-trivial, Lemma~\ref{lem:ekr-concave-nontriv} gives a strictly smaller norm for all sufficiently large $n$.  If $\cH$ is contained in a full $k$-star, then coordinatewise degree comparison gives the result because $x\mapsto x^p$ is strictly increasing.  Equality forces every edge of the full $k$-star to be present.
\end{proof}

\section{Expansion paths}\label{sec:paths}

\subsection{Auxiliary statements}\label{sec:path-aux}

This section contains the auxiliary tools used in the proof of Theorem~\ref{thm:path-asymptotic-stability}.  We avoid using any unproved stability statement for rich-pair or rich-vertex auxiliary graphs.  The case $t\ge3$ uses high $t$-shadows and Theorem~\ref{thm:kmv-edge-stability}; the case $t=2$ uses a weighted high-pair argument; and the case $t=1$ uses a weighted high-vertex argument.  In the exact step we use two cleaning consequences of Kostochka, Mubayi and Verstra\"ete~\cite[Section~6.4]{KostochkaMubayiVerstraete2015}.

\begin{lem}\label{lem:constant-expansion}
Fix $r\ge3$, $2\le t\le r-1$, and $\ell\ge4$.  There is a constant $C_{\rm ext}=C_{\rm ext}(r,t,\ell)$ with the following property.  Let $n$ be sufficiently large and let $\cH\subseteq\binom{[n]}r$.  Let $T_1,\ldots,T_\ell$ be the edges of a $t$-uniform linear path $P_\ell^t$.  If
\[
   d_\cH(T_i)>C_{\rm ext}n^{r-t-1}
   \qquad\text{for every }i\in[\ell],
\]
then $\cH$ contains a copy of $P_\ell^r$ whose $i$th edge contains $T_i$.
\end{lem}

\begin{proof}
Let $U=\bigcup_{i=1}^\ell T_i$.  We choose edges $E_i\in\cH$ with $T_i\subseteq E_i$ one at a time.  Suppose that $E_1,\ldots,E_{i-1}$ have already been chosen and satisfy the required linear intersection pattern.  Let $W$ be the set of all vertices that have already been used, together with the vertices of $U$.  Since $r,t,\ell$ are fixed, $|W|=O_{r,t,\ell}(1)$.

An extension $E\supseteq T_i$ is forbidden if it contains a vertex of $W\setminus T_i$ which would create an intersection not prescribed by the path, or if it creates any other unintended intersection with a previously chosen edge.  For each specified vertex $v\notin T_i$, the number of $r$-sets containing $T_i\cup\{v\}$ is
\[
   \binom{n-t-1}{r-t-1}=O(n^{r-t-1}).
\]
There are only $O_{r,t,\ell}(1)$ specified vertices and finitely many further forbidden intersection patterns, all contributing at most the same order.  Hence the total number of forbidden extensions is at most $C_{\rm ext}n^{r-t-1}$ for a constant depending only on $r,t,\ell$.  Since $d_\cH(T_i)>C_{\rm ext}n^{r-t-1}$, an admissible extension remains.  Repeating this for all $i$ gives a linear $r$-path.
\end{proof}
We need the following lemma in graph case.
\begin{lem}\label{lem:graph-local-high-pairs}
Fix integers $r\ge3$ and $\ell\ge4$.  Let $G$ be an $n$-vertex graph containing no ordinary path with $\ell$ edges.  Then the number of sets $S\in\binom{V(G)}r$ such that
\[
   e_G(S)>r-1
\]
is $O(n^{r-2})$, where the constant depends only on $r$ and $\ell$.
\end{lem}

\begin{proof}
By the Erd\H{o}s--Gallai theorem, every $P_\ell$-free graph has at most $(\ell-1)n/2$ edges.  Then, every subgraph of $G$ has a vertex with degree at most $\ell-1$.  
Thus, $G$ is $(\ell-1)$-degenerate. 

Moreover, the number of triangles in a $d$-degenerate graph is at most $O_d(n)$: orient every edge from earlier to later in a degeneracy ordering; each vertex has at most $d$ out-neighbours, so each triangle is counted at its earliest vertex and there are at most $\binom d2 n$ such choices.  Hence $G$ contains $O(n)$ triangles.  Also $e(G)=O(n)$, so the number of unordered pairs of disjoint edges in $G$ is $O(n^2)$.

If a graph on an $r$-set $S$ has more than $r-1$ edges, then it contains either a triangle or two vertex-disjoint edges.  Indeed, a triangle-free graph with no two disjoint edges has all its edges incident with one common vertex, and therefore has at most $r-1$ edges on $r$ vertices.  Therefore every $S$ counted in the lemma contains a triangle or a copy of $2K_2$.  The number of $r$-sets containing a triangle is
\[
   O(n)\binom n{r-3}=O(n^{r-2}),
\]
and, for $r\ge4$, the number of $r$-sets containing two disjoint edges is
\[
   O(n^2)\binom n{r-4}=O(n^{r-2});
\]
for $r=3$ this latter case is impossible.  This proves the lemma.
\end{proof}

\begin{lem}\label{lem:weighted-pair-bound}
Fix $r\ge3$, $\ell\ge4$, and $p>1$. 
Let $D_2=\binom{n-2}{r-2}$.  For every fixed $\gamma>0$, if $\cH\subseteq\binom{[n]}r$ is $P_\ell^r$-free, then for sufficiently large $n$,
\[
   \norm{\cH}_{2,p}
   \le
   \left(r-1+\binom r2\gamma^{p-1}\right)D_2^{p-1}|\cH|
   +o(nD_2^p).
\]
\end{lem}

\begin{proof}
Define the \emph{high-pair graph} $G_\gamma$ on $[n]$ by declaring that $xy\in E(G_\gamma)$ if and only if $d_\cH(xy)\ge \gamma \binom{n-2}{r-2}$.
According to Lemma \ref{lem:constant-expansion} $G_\gamma$ is $P_\ell$-free.
For an edge $E\in\cH$, let
\[
   w(E)=\sum_{xy\subseteq E} d_\cH(xy)^{p-1}.
\]
Then
\[
   \norm{\cH}_{2,p}=\sum_{xy}d_\cH(xy)^p=\sum_{E\in\cH}w(E).
\]
Let
\[
   \cB_\gamma=\{E\in\cH:e_{G_\gamma}(E)>r-1\}.
\]
By Lemma~\ref{lem:graph-local-high-pairs}, $|\cB_\gamma|=O(n^{r-2})$.  Since $w(E)\le\binom r2\binom{n-2}{r-2}^{p-1}$ for every $E$, the total contribution of $\cB_\gamma$ is
\[
   O\left(n^{r-2}D_2^{p-1}\right)
   =O\left(n^{p(r-2)}\right)
   =o(nD_2^p).
\]

If $E\notin\cB_\gamma$, then $E$ contains at most $r-1$ high pairs.  These contribute at most $(r-1)\binom{n-2}{r-2}^{p-1}$ to $w(E)$.  The remaining pairs have degree less than $\gamma \binom{n-2}{r-2}$, and hence contribute at most $\binom r2(\gamma \binom{n-2}{r-2})^{p-1}$.  Thus
\[
   w(E)\le \left(r-1+\binom r2\gamma^{p-1}\right)\binom{n-2}{r-2}^{p-1}
\]
for every $E\notin\cB_\gamma$.  Summing over $E\in\cH$ and adding the exceptional contribution proves the lemma.
\end{proof}

\begin{lem}\label{lem:weighted-vertex-bound}
Fix $r\ge3$, $\ell\ge4$, and $p>1$.
Let $D_1=\binom{n-1}{r-1}$.  For every fixed $0<\gamma\le1$, if $\cH\subseteq\binom{[n]}r$ is $P_\ell^r$-free, then
\[
   \norm{\cH}_{1,p}
   \le
   \left(1+(r-1)\gamma^{p-1}\right)D_1^{p-1}|\cH|
   +o(D_1^p),
\]
where the implicit $o(\cdot)$ term may depend on $r,\ell,p$ and $\gamma$.
\end{lem}

\begin{proof}
Let
\[
   U_\gamma=\{v\in[n]:d_\cH(v)\ge \gamma D_1\}.
\]
By Theorem~\ref{thm:kmv-edge-stability}, $|\cH|=O(n^{r-1})$. Hence
\[
   \gamma D_1 |U_\gamma|
   \le \sum_{v\in U_\gamma}d_\cH(v)
   \le r|\cH|,
\]
so $|U_\gamma|=O_\gamma(1)$.

For $E\in\cH$, define
\[
   w(E)=\sum_{v\in E}d_\cH(v)^{p-1}.
\]
Then
\[
   \norm{\cH}_{1,p}=\sum_{v\in[n]}d_\cH(v)^p=\sum_{E\in\cH}w(E).
\]
Let $\cB_\gamma$ be the set of edges containing at least two vertices of $U_\gamma$. Since $|U_\gamma|=O_\gamma(1)$,
\[
   |\cB_\gamma|\le \binom{|U_\gamma|}{2}\binom n{r-2}+O(n^{r-3})=O_\gamma(n^{r-2}).
\]
Also $w(E)\le rD_1^{p-1}$ for every $E$, and therefore the total contribution of $\cB_\gamma$ is
\[
   O_\gamma\left(n^{r-2}D_1^{p-1}\right)=O_\gamma(n^{p(r-1)-1}).
\]

If $E\notin\cB_\gamma$, then $E$ contains at most one vertex of $U_\gamma$. A possible high vertex contributes at most $D_1^{p-1}$ to $w(E)$, while all remaining vertices have degree less than $\gamma D_1$. Since $0<\gamma\le1$, this gives
\[
   w(E)\le \left(1+(r-1)\gamma^{p-1}\right)D_1^{p-1}.
\]
Summing over $E\in\cH\setminus\cB_\gamma$ and adding the exceptional contribution proves the lemma.
\end{proof}

\begin{lem}\label{lem:power-comparison}
Let $p>1$ and $D>0$.
\begin{enumerate}[label=\textup{(\roman*)}]
\item If $0\le x\le x+y\le D$, then
\[
   (x+y)^p-x^p\le pD^{p-1}y.
\]
\item If $0\le m\le D$, then
\[
   D^p-(D-m)^p\ge D^{p-1}m.
\]
\end{enumerate}
\end{lem}

\begin{proof}
For (i),
\[
   (x+y)^p-x^p=\int_x^{x+y}pu^{p-1}\,du\le pD^{p-1}y.
\]
For (ii), let $z=m/D$.  Since $0\le z\le1$ and $p>1$, we have $(1-z)^p\le1-z$.  Hence
\[
   D^p-(D-m)^p=D^p\left(1-(1-z)^p\right)\ge D^p z=D^{p-1}m.
\]
\end{proof}

The following two lemmas are the stability-to-exact ingredients for expansion paths.  They are the forms of the cleaning step in the exact proof of Kostochka, Mubayi and Verstra\"ete~\cite[Section~6.4]{KostochkaMubayiVerstraete2015} needed for the norm comparison below.

\begin{lem}[Kostochka--Mubayi--Verstra\"ete~\cite{KostochkaMubayiVerstraete2015}]\label{lem:odd-defect}
Let $r\ge3$ and $a\ge2$.  For every $\eta>0$ there are $\alpha>0$ and $n_0$ such that the following holds for all $n\ge n_0$.  Let $A\in\binom{[n]}a$ and let $\cH\subseteq\binom{[n]}r$ be $P_{2a+1}^r$-free.  Define
\[
   \cB=\{H\in\cH:H\cap A=\emptyset\},\qquad
   \cM=\cH_{n,r,A}\setminus\cH.
\]
If $|\cB|+|\cM|\le \alpha n^{r-1}$ and $\cB\ne\emptyset$, then
\[
   |\cB|\le \eta |\cM|.
\]
\end{lem}

\begin{lem}[Kostochka--Mubayi--Verstra\"ete \cite{KostochkaMubayiVerstraete2015}]\label{lem:even-cleaning}
Let $r\ge3$, $s\ge2$, and $a=s-1$.  For every $\eta>0$ there are $\alpha>0$ and $n_0$ such that the following holds for all $n\ge n_0$.  Let $A\in\binom{[n]}a$ and let $\cH\subseteq\binom{[n]}r$ be $P_{2s}^r$-free.  Define
\[
   \cB=\{H\in\cH:H\cap A=\emptyset\},\qquad
   \cM=\cH_{n,r,A}\setminus\cH.
\]
If $|\cB|+|\cM|\le \alpha n^{r-1}$, then there is a $2$-intersecting subfamily $\cB_0\subseteq\cB$ such that
\[
   |\cB\setminus\cB_0|\le \eta |\cM|.
\]
Moreover, $\cH_{n,r,A}\cup\cB_0$ is $P_{2s}^r$-free.
\end{lem}

\subsection{Asymptotic value and stability}\label{sec:path-asymptotic-proof}

First, we prove the asymptotic value and stability for the $(t,p)$-norm.
\begin{thm}\label{thm:path-asymptotic-stability}
   Let $r\ge3$, $\ell\ge4$, $1\le t\le r-1$, and $p>1$ be fixed.  Let $a=\left\lfloor(\ell-1)/2\right\rfloor$ and $D_t=\binom{n-t}{r-t}$.  Then
   \[
      \max\left\{\norm{\cH}_{t,p}:\cH\subseteq\binom{[n]}r\text{ is }P_\ell^r\text{-free}\right\}
      =(a+o(1))\binom n{t-1}D_t^p .
   \]
   Equivalently, if $\ell=2s+1$ is odd, the coefficient is $s$, and if $\ell=2s$ is even, the coefficient is $s-1$.
   
   Moreover, the asymptotic extremal families are stable in the edge metric.  If $\cH\subseteq\binom{[n]}r$ is $P_\ell^r$-free and
   \[
      \norm{\cH}_{t,p}\ge (a-o(1))\binom n{t-1}D_t^p,
   \]
   then there exists an $a$-set $A\subseteq[n]$ such that
   \[
      |\cH\triangle \cH_{n,r,A}|=o(n^{r-1}).
   \]
   \end{thm}

We prove Theorem~\ref{thm:path-asymptotic-stability}.  Throughout this section $r\ge3$, $\ell\ge4$, $p>1$, and $1\le t\le r-1$ are fixed; for convenience, let $a=\left\lfloor(\ell-1)/2\right\rfloor$ and $D_t=\binom{n-t}{r-t}$.

\subsubsection{Lower bound}

Let $A\subseteq[n]$ with $|A|=a$, and take the full $a$-star
\[
   \cS_A=\cH_{n,r,A}=\{E\in\binom{[n]}r:E\cap A\ne\emptyset\}.
\]
The ordinary path with $\ell$ edges has vertex-cover number $\lceil \ell/2\rceil=a+1$, while every edge of $\cS_A$ is covered by $A$.  Hence $\cS_A$ is $P_\ell^r$-free.

If $T\in\binom{[n]}t$ meets $A$, then $d_{\cS_A}(T)=D_t$.  The number of such $T$ is
\[
   (a+o(1))\binom n{t-1}.
\]
If $T\cap A=\emptyset$, then
\[
   d_{\cS_A}(T)=D_t-\binom{n-a-t}{r-t}=O(n^{r-t-1}).
\]
Therefore
\[
   \sum_{T\cap A=\emptyset}d_{\cS_A}(T)^p
   =O(n^t n^{p(r-t-1)})
   =o\left(\binom n{t-1}D_t^p\right),
\]
because $p>1$.  Thus
\[
   \norm{\cS_A}_{t,p}=(a+o(1))\binom n{t-1}D_t^p.
\]
This proves the lower bound.

\subsubsection{Upper bound and stability in the high-shadow range}

Assume first that $3\le t\le r-1$.  Let $\omega>C_{\rm ext}(r,t,\ell)$ be a fixed constant and define the \emph{high $t$-shadow}
\[
   \cG=\left\{T\in\binom{[n]}t:d_\cH(T)>\omega n^{r-t-1}\right\}.
\]
The contribution of $t$-sets outside $\cG$ is at most
\[
   \binom nt(\omega n^{r-t-1})^p
   =O(n^{t+p(r-t-1)})
   =o\left(\binom n{t-1}D_t^p\right).
\]
By Lemma~\ref{lem:constant-expansion}, $\cG$ is $P_\ell^t$-free.  Theorem~\ref{thm:kmv-edge-stability} gives
  $ |\cG|\le (a+o(1))\binom n{t-1}.$

Therefore
\[
   \norm{\cH}_{t,p}
   \le |\cG|D_t^p+o\left(\binom n{t-1}D_t^p\right)
   \le (a+o(1))\binom n{t-1}D_t^p.
\]
This proves the upper bound for $t\ge3$.

Now suppose that $\cH$ is near-extremal in norm:
\[
   \norm{\cH}_{t,p}\ge (a-o(1))\binom n{t-1}D_t^p.
\]
The preceding low-degree estimate implies
\[
   \sum_{T\in\cG}d_\cH(T)^p\ge (a-o(1))\binom n{t-1}D_t^p.
\]
Since $d_\cH(T)\le D_t$ for all $T$, we get
\[
   |\cG|\ge (a-o(1))\binom n{t-1}.
\]
By the stability part of Theorem~\ref{thm:kmv-edge-stability}, there is an $a$-set $A\subseteq[n]$ such that
\[
   |\{T\in\cG:T\cap A=\emptyset\}|=o(n^{t-1}).
\]
Since the family of all $t$-sets meeting $A$ has size $(a+o(1))\binom n{t-1}$, it follows that all but $o(n^{t-1})$ of the $t$-sets meeting $A$ are in $\cG$.

Moreover, near equality also forces a small total degree deficit on these high sets.  Indeed,
\[
   \sum_{T\in\cG}(D_t^p-d_\cH(T)^p)=o\left(\binom n{t-1}D_t^p\right).
\]
By Lemma~\ref{lem:power-comparison}(ii),
\[
   \sum_{\substack{T\in\cG\\T\cap A\ne\emptyset}}(D_t-d_\cH(T))=o(n^{t-1}D_t)=o(n^{r-1}).
\]
The $o(n^{t-1})$ sets meeting $A$ but not in $\cG$ contribute at most $o(n^{t-1})D_t=o(n^{r-1})$ to the same deficit.  Hence
\[
   \sum_{\substack{T\in\binom{[n]}t\\T\cap A\ne\emptyset}}(D_t-d_\cH(T))=o(n^{r-1}).
\]
Each missing edge of the full $A$-star contains at least $\binom{r-1}{t-1}$ $t$-sets meeting $A$.  Therefore
\[
   |\cH_{n,r,A}\setminus\cH|=o(n^{r-1}).
\]
The edge bound in Theorem~\ref{thm:kmv-edge-stability} gives
\[
   |\cH|\le (a+o(1))\binom n{r-1}=|\cH_{n,r,A}|+o(n^{r-1}).
\]
Since $\cH$ contains all but $o(n^{r-1})$ edges of $\cH_{n,r,A}$, it follows that
\[
   |\cH\setminus\cH_{n,r,A}|=o(n^{r-1}).
\]
Thus $|\cH\triangle\cH_{n,r,A}|=o(n^{r-1})$.

\subsubsection{Upper bound and stability in the pair case}

Now let $t=2$ and set $D_2=\binom{n-2}{r-2}$.
Fix $\gamma>0$.  Lemma~\ref{lem:weighted-pair-bound} and Theorem~\ref{thm:kmv-edge-stability} give
\[
\begin{aligned}
   \norm{\cH}_{2,p}
   &\le
   \left(r-1+\binom r2\gamma^{p-1}\right)D_2^{p-1}|\cH|+o(nD_2^p)\\
   &\le
   \left(r-1+\binom r2\gamma^{p-1}\right)D_2^{p-1}(a+o(1))\binom n{r-1}+o(nD_2^p).
\end{aligned}
\]
Since
\[
   (r-1)\binom n{r-1}D_2^{p-1}=(1+o(1))nD_2^p,
\]
we have
\[
   \norm{\cH}_{2,p}\le (a+O(\gamma^{p-1})+o(1))nD_2^p.
\]
Letting first $n\to\infty$ and then $\gamma\to0$ gives
\[
   \norm{\cH}_{2,p}\le (a+o(1))nD_2^p.
\]
Together with the lower bound, this proves the asymptotic value for $t=2$.

It remains to prove stability.  Suppose
\[
   \norm{\cH}_{2,p}\ge (a-o(1))nD_2^p.
\]
Combining this lower bound with Lemma~\ref{lem:weighted-pair-bound}, we obtain, for every fixed $\gamma>0$,
\[
   (a-o(1))nD_2^p
   \le
   \left(r-1+\binom r2\gamma^{p-1}\right)D_2^{p-1}|\cH|+o(nD_2^p).
\]
After dividing by $D_2^{p-1}$ and using
\[
   nD_2=(1+o(1))(r-1)\binom n{r-1},
\]
we get
\[
   |\cH|
   \ge
   \left(a\cdot\frac{r-1}{r-1+\binom r2\gamma^{p-1}}-o(1)\right)
   \binom n{r-1}
   =
   (a-O(\gamma^{p-1})-o(1))\binom n{r-1}.
\]
Since $\gamma>0$ is arbitrary, this implies
\[
   |\cH|\ge (a-o(1))\binom n{r-1}.
\]
The stability part of Theorem~\ref{thm:kmv-edge-stability} now gives an $a$-set $A\subseteq[n]$ such that
\[
   |\cH\triangle\cH_{n,r,A}|=o(n^{r-1}).
\]
This proves the stability assertion for $t=2$.

\subsubsection{Upper bound and stability in the vertex case}

Finally, let $t=1$ and set $D_1=\binom{n-1}{r-1}$.
The lower bound from the full $a$-star is
\[
   \norm{\cH_{n,r,A}}_{1,p}=(a+o(1))D_1^p,
\]
because the vertices in $A$ have degree $D_1$, while every vertex outside $A$ has degree
\[
   D_1-\binom{n-a-1}{r-1}=O(n^{r-2}),
\]
and $n\cdot n^{p(r-2)}=o(D_1^p)$ since $p>1$.

For the upper bound, fix $0<\gamma\le1$. By Lemma~\ref{lem:weighted-vertex-bound} and Theorem~\ref{thm:kmv-edge-stability},
\[
\begin{aligned}
   \norm{\cH}_{1,p}
   &\le \left(1+(r-1)\gamma^{p-1}\right)D_1^{p-1}|\cH|+o(D_1^p)\\
   &\le \left(1+(r-1)\gamma^{p-1}\right)(a+o(1))D_1^p+o(D_1^p).
\end{aligned}
\]
Letting first $n\to\infty$ and then $\gamma\to0$ gives
\[
   \norm{\cH}_{1,p}\le (a+o(1))D_1^p.
\]
Together with the lower bound, this proves the asymptotic value for $t=1$.

Now assume that
\[
   \norm{\cH}_{1,p}\ge (a-o(1))D_1^p.
\]
Using Lemma~\ref{lem:weighted-vertex-bound}, we obtain, for every fixed $0<\gamma\le1$,
\[
   (a-o(1))D_1^p
   \le \left(1+(r-1)\gamma^{p-1}\right)D_1^{p-1}|\cH|+o(D_1^p).
\]
Hence
\[
   |\cH|
   \ge
   \left(\frac{a}{1+(r-1)\gamma^{p-1}}-o(1)\right)D_1.
\]
Since $\gamma$ is arbitrary and $D_1=(1+o(1))\binom n{r-1}$, it follows that
\[
   |\cH|\ge (a-o(1))\binom n{r-1}.
\]
The stability part of Theorem~\ref{thm:kmv-edge-stability} gives an $a$-set $A\subseteq[n]$ such that
\[
   |\cH\triangle\cH_{n,r,A}|=o(n^{r-1}).
\]
This proves the stability assertion for $t=1$ and completes the proof of Theorem~\ref{thm:path-asymptotic-stability}.

\subsection{From stability to exact results}\label{sec:path-exact}

We now prove Theorems~\ref{thm:exact-odd-path} and~\ref{thm:exact-even-path}.  The argument is the standard stability-to-exact comparison: edges outside the stable star can increase the norm only by a controlled amount, while the missing edges of the star cause a larger loss.  In the even case the surviving outside part is allowed, but it must be $2$-intersecting and is then optimized by the background EKR theorem, Theorem~\ref{thm:ekr-convex}, with $k=2$.

\subsubsection{Odd path length}

\begin{proof}[Proof of Theorem~\ref{thm:exact-odd-path}]
Let $\cH$ be an extremal $P_{2a+1}^r$-free family for the $(t,p)$-norm.  The full $a$-star $\cH_{n,r,a}$ is $P_{2a+1}^r$-free, so extremality and Lemma~\ref{lem:vertex-star-norm} give
\[
   \norm{\cH}_{t,p}\ge \norm{\cH_{n,r,a}}_{t,p}
   =(a+o(1))\binom n{t-1}\binom{n-t}{r-t}^p.
\]
By Theorem~\ref{thm:path-asymptotic-stability}, there is an $a$-set $A\subseteq[n]$ such that, with
\[
   \cB=\{H\in\cH:H\cap A=\emptyset\},\qquad
   \cM=\cH_{n,r,A}\setminus\cH,
\]
we have
\[
   |\cB|+|\cM|=o(n^{r-1}).
\]
Set $D_t=\binom{n-t}{r-t}$.
We compare $\cH$ with the full star $\cH_{n,r,A}$.  For $T\cap A\ne\emptyset$, outside edges in $\cB$ do not contribute to $d_\cH(T)$.  If
we write $m_T=|\{M\in\cM:T\subseteq M\}|$, then
the degree of $T$ has been reduced from $D_t$ to $D_t-m_T$.  By Lemma~\ref{lem:power-comparison}(ii), the total loss on the $t$-sets meeting $A$ is at least
\[
   \sum_{T\cap A\ne\emptyset}\left(D_t^p-(D_t-m_T)^p\right)
   \ge D_t^{p-1}\sum_{T\cap A\ne\emptyset}m_T.
\]
Every missing star edge contains at least $\binom{r-1}{t-1}$ $t$-subsets meeting $A$.  Hence this loss is at least
\begin{equation}\label{eq:odd-loss-exact}
   \binom{r-1}{t-1}D_t^{p-1}|\cM|.
\end{equation}

The possible gain from the outside edges is bounded as follows.  For every $T$ disjoint from $A$, let $b_T=|\{B\in\cB:T\subseteq B\}|$.
Before adding the outside edges, the degree of $T$ is at most $D_t-b_T$, and after adding them it is at most $D_t$.  Lemma~\ref{lem:power-comparison}(i) gives
\[
   (x+b_T)^p-x^p\le pD_t^{p-1}b_T
\]
for the relevant value of $x$.  Summing over all $T$ gives a total outside gain at most
\begin{equation}\label{eq:odd-gain-exact}
   pD_t^{p-1}\sum_T b_T
   =p\binom rtD_t^{p-1}|\cB|.
\end{equation}
Choose $\eta>0$ such that
\[
   \eta<\frac{\binom{r-1}{t-1}}{p\binom rt}.
\]
For all sufficiently large $n$, Lemma~\ref{lem:odd-defect} applies.  If $\cB\ne\emptyset$, then $|\cB|\le\eta|\cM|$, and the gain in \eqref{eq:odd-gain-exact} is strictly smaller than the loss in \eqref{eq:odd-loss-exact}.  This contradicts the extremality of $\cH$, because the full $A$-star is admissible.  Therefore $\cB=\emptyset$.

Thus $\cH\subseteq\cH_{n,r,A}$.  If this inclusion is proper, then at least one $t$-degree is strictly smaller than in the full star and no $t$-degree is larger, so the norm is strictly smaller.  Hence $\cH=\cH_{n,r,A}$, as required.
\end{proof}

\subsubsection{Even path length}

\begin{proof}[Proof of Theorem~\ref{thm:exact-even-path}]
Let $a=s-1$, and fix an admissible family $\cE_{n,r,s}(A,Q)$ as in \eqref{eq:even-extremal-family}.  This family is $P_{2s}^r$-free.  Indeed, any two outside edges both contain $Q$ and therefore cannot both lie in a linear path.  Thus a linear path in $\cE_{n,r,s}(A,Q)$ contains at most one outside edge.  The remaining $2s-1$ edges would all have to meet the $(s-1)$-set $A$, but in a linear path a fixed vertex can lie in at most two consecutive edges.  Hence $A$ can cover at most $2s-2$ edges, a contradiction.

Let $\cH$ be an extremal $P_{2s}^r$-free family for the $(t,p)$-norm.  Since $\cE_{n,r,s}$ is admissible and differs from the full $(s-1)$-star only by $O(n^{r-2})$ outside edges, its norm is
\[
   \norm{\cE_{n,r,s}}_{t,p}
   =(s-1+o(1))\binom n{t-1}\binom{n-t}{r-t}^p.
\]
Therefore Theorem~\ref{thm:path-asymptotic-stability} gives an $(s-1)$-set $A\subseteq[n]$ such that, with
\[
   \cB=\{H\in\cH:H\cap A=\emptyset\},\qquad
   \cM=\cH_{n,r,A}\setminus\cH,
\]
we have
\[
   |\cB|+|\cM|=o(n^{r-1}).
\]
Again write $D_t=\binom{n-t}{r-t}$.

Choose $\eta>0$ so small that
\[
   p\binom rt\eta<\frac12\binom{r-1}{t-1}.
\]
For sufficiently large $n$, Lemma~\ref{lem:even-cleaning} gives a $2$-intersecting subfamily $\cB_0\subseteq\cB$ such that
\[
   |\cB\setminus\cB_0|\le \eta|\cM|
\]
and such that $\cH_{n,r,A}\cup\cB_0$ is $P_{2s}^r$-free.
We compare $\cH$ with $\cH':=\cH_{n,r,A}\cup\cB_0$.
Passing from $\cH$ to $\cH'$ adds all missing star edges in $\cM$ and removes the outside edges in $\cB\setminus\cB_0$.  The added star edges increase the norm by at least
\[
   \binom{r-1}{t-1}D_t^{p-1}|\cM|,
\]
by the same loss estimate as in the odd case.  The removed outside edges can decrease the norm by at most
\[
   p\binom rtD_t^{p-1}|\cB\setminus\cB_0|
   \le p\binom rt\eta D_t^{p-1}|\cM|.
\]
By the choice of $\eta$, this is less than half of the gain from adding $\cM$ whenever $\cM\ne\emptyset$.  Hence extremality forces
\[
   \cM=\emptyset,
   \qquad
   \cB=\cB_0.
\]
Thus
\[
   \cH=\cH_{n,r,A}\cup\cB,
\]
where $\cB\subseteq\binom{[n]\setminus A}r$ is $2$-intersecting.

It remains to optimize the outside part.  Let $N=n-|A|=n-s+1$.  The contribution of $t$-sets meeting $A$ is already fixed and equal to that of the full $A$-star.  For every $t$-set $T\subseteq[n]\setminus A$, the background degree supplied by the full $A$-star is the constant $c_t=D_t-\binom{N-t}{r-t}=O(N^{r-t-1})$.
Therefore the part of the norm depending on $\cB$ is exactly
\[
   \sum_{T\in\binom{[n]\setminus A}t}\left(c_t+d_\cB(T)\right)^p
   =\Phi_{c_t,t}(\cB)
\]
on the $N$-vertex ground set $[n]\setminus A$.  By Theorem~\ref{thm:ekr-convex} with $k=2$, this functional is uniquely maximized over $2$-intersecting outside families by a full $2$-star.

Thus
\[
   \cB=\{E\in\binom{[n]\setminus A}r:Q\subseteq E\}
\]
for some pair $Q\subseteq[n]\setminus A$.  Consequently $\cH\cong\cE_{n,r,s}$, and \eqref{eq:even-extremal-value} gives the exact value.
\end{proof}
\section{Concluding remarks}

We have determined the maximum $(t,p)$-norm in three extremal settings for all sufficiently large $n$.  For bounded matching number and for $k$-intersecting families, the convex results hold in the stronger background form $\sum_T(c+d(T))^p$, while the concave zero-background results are exact for every $1\le t\le r-1$.  For expansion paths, the convex range $p>1$ is exact for every $1\le t\le r-1$.  The proof separates the shadow dimensions: high $t$-shadows handle $t\ge3$, a weighted high-pair argument handles $t=2$, and a weighted high-vertex argument handles $t=1$.  These low-dimensional arguments are useful because they avoid separate stability theorems for auxiliary rich-pair or rich-vertex graphs.

With a similar argument, one can give the maximum value of the $(t,p)$-norm for $C_{\ell}^r$-free hypergraphs.
For seek of simplicity, we leave it as future work.

\section*{Acknowledgments}

The author thanks D\'{a}niel Gerbner for helpful discussions and valuable suggestions.

\section*{Declaration on the use of artificial intelligence}

OpenAI's GPT-5.5 was used during the initial exploration of this project.  In particular, it suggested the statement of Lemma~\ref{lem:shift-norm}, which simplified the proof in Section~\ref{sec:matching}.  It also assisted with some inequality calculations, with drafting parts of the proofs, for example, Lemma~\ref{lem:power-comparison}, and with general writing.  The authors have independently checked all arguments and are solely responsible for the correctness and content of the paper.

\bibliography{ref.bib}

\end{document}